\documentclass[12pt]{article}
\usepackage{amssymb}
\usepackage{graphicx, enumerate}
\usepackage{amsfonts}
\usepackage{graphicx}
\usepackage{amsmath}
\usepackage{amsthm}
\usepackage{amssymb}
\usepackage[colorlinks]{hyperref}  
\usepackage{color}

    \newtheorem{theorem}{Theorem}
    \newtheorem{lemma}[theorem]{Lemma}
    \newtheorem{proposition}[theorem]{Proposition}

\theoremstyle{definition} 

\newcommand{\eps}{\varepsilon}

\newcommand{\lstar}{{\raise-0.15ex\hbox{$\scriptstyle \ast$}}}

\theoremstyle{remark} 

\newcommand{\rr}{\mathbb{R}}
\newcommand{\ind}{{\bf 1}}

\newcommand{\re}{\textup{Re}}
\newcommand{\sech}{\textup{ sech }}

\newcommand{\Sineb}{\textup{Sine}_\beta}

\newcommand{\wt}{\widetilde}

\newcommand{\Psine}{\tilde{ \mathcal{J}}_{\textup{Sine}}}
\newcommand{\ldp}{\mathcal{I}}

\newcommand{\II}{\mathcal{H}}
\newcommand{\sineldp}{I_{\text{Sine}}}

\newcommand{\ed}{\stackrel{d}{=}}
\newcommand{\os}{\text{Osc}}

\newcommand{\Bess}{\textup{Bess}_{a,\beta}}
\newcommand{\Pbess}{\tilde{ \mathcal{J}}_{\Bess}}
\newcommand{\bessldp}{I_{\Bess}}
\newcommand{\Pathbess}{\mathcal{J}_{\Bess}}
\newcommand{\Pathtilde}{\mathcal{J}_{\tilde \varphi \text{ Path}}}
\newcommand{\im}{\textup{Im}}

\begin{document}

\title{The random matrix hard edge: rare events and a transition}%

\author{Diane Holcomb\footnote{Department of Mathematics, KTH,  holcomb@kth.se} }%




\maketitle

\begin{abstract}
We study properties of the point process that appears as the local limit at the random matrix hard edge. We show a transition from the hard edge to bulk behavior and give a central limit theorem and large deviation result for the number of points in a growing interval $[0,\lambda]$ as $\lambda \to \infty$. We study these results for the square root of the hard edge process. In this setting many of these behaviors mimic those of the $\Sineb$ process.

\end{abstract}

\section{Introduction}

In the study of classical Hermitian random matrix ensembles three distinct types of local behavior have been observed. The Gaussian Unitary Ensemble (GUE) exhibits one type of behavior in the interior, or bulk, of its spectrum and another at its edge. Scaling the $n \times n$ model and passing to the $n\to \infty$ limit one obtains the Sine$_2$ and Airy$_2$  processes respectively. The Laguerre (also called Wishart) and Jacobi (MANOVA) ensembles exhibit the same behavior in the bulk, but depending on the choice of parameters may exhibit  two different types of behavior at the edge. In one case the limit process at the edge is again the Airy$_2$ process. This is referred to as a soft edge. In the other case the limit process at the edge is a family of determinantal point processes where the determinant is defined in terms of Bessel functions $J_\alpha$,

\[
K_\alpha (x,y) = \frac{J_\alpha(\sqrt{x})\sqrt y J'_\alpha(\sqrt{y}) - \sqrt{x} J'_\alpha(\sqrt x)J_\alpha(\sqrt y)}{2(x-y)}.
\]  
This type of limiting behavior will occur when the eigenvalues of the random matrix are pushed against some hard constraint, and so will be referred to as hard-edge behavior.

The Laguerre and Jacobi ensembles may be generalized to one parameter families of point processes called $\beta$-ensembles defined through their joint distribution. In particular the $\beta$-Laguerre ensemble has joint density 
\begin{equation}
\label{wishart}
 p_{n,m,\beta}(\lambda_1, \lambda_2, \dots, \lambda_n)= \frac{1}{Z_{\beta,n,m}}\prod_{i=1}^n \lambda_i^{\frac{\beta}{2}(m-n+1)-1} e^{-\frac{\beta}{2} \lambda_i} \prod_{j<k} |\lambda_j-\lambda_k|^\beta.
\end{equation}
Here $\beta$ may be any value greater than $0$, $m \geq n$,  and $Z_{\beta,n,m}$ is an explicitly computable normalizing constant. With a slight abuse of terminology we refer to the points of the $\beta$-Laguerre ensemble as its eigenvalues. The local limits of these $\beta$-ensembles may again be studied, but can no longer be described by a determinantal point process.

As in the classical case, the $\beta$-Laguerre ensemble can exhibit two different types of limiting behavior at the lower edge of the spectrum. For $m/n \to \gamma \neq  1$ the lower edge of the spectrum exhibits soft-edge behavior.  In this case the appropriately rescaled lower edge of the $\beta$-Laguerre ensemble converges to the Airy$_\beta$ process \cite{RRV}. In the case where instead $m = n + a_n$ and $a_n \to a$ the lower edge of the spectrum exhibits hard-edge behavior \cite{RR}. In the intermediate regime where $a_n\to \infty, a_n/n\to 0$ it is expected that the behavior is soft edge and so the limiting process will be Airy$_\beta$.  For this regime there is a partial result in the case $\beta=2$ for $a_n \sim c\sqrt{n}$ by Deift, Menon, and Trogdon (see \cite{DMT}), but otherwise the problem remains open. Similar soft and hard edge scaling results were shown for the $\beta$-Jacobi ensemble \cite{HMF}. Later universality results extended the soft edge limit to a wide class of $\beta$-ensembles \cite{bour2,KRV}, and recent work by Rider and Waters did the same for the hard edge \cite{BRPW}.

Let $\lambda_0 < \lambda_1< \lambda_2< ...$ be the ordered eigenvalues of the $\beta$-Laguerre ensemble. For the hard edge regime when $a_n \to a$ the set $\{n \lambda_0,n \lambda_1,  ..., n \lambda_k\}$ converges to the first $k$ eigenvalues of the Stochastic Bessel operator introduced by Ram\'irez and Rider in \cite{RR}.  The operator acts on functions  $\rr_+ \to \rr$ and is given by:
\[\mathfrak{G}_{\beta,a}=-\exp \left[ (a+1)x+ \frac{2}{\sqrt{\beta}} b(x) \right] \cdot \frac{d}{dx} \left( \exp\left[-ax- \frac{2}{\sqrt{\beta}}b(x)\right]\frac{d}{dx}\right),\]
with Dirichlet boundary conditions at 0 and Neumann conditions at infinity, where $b(x)$ is a Brownian motion, $a>-1$ and $\beta>0$.  Moreover, it can be shown that the spectrum defines a simple point process which will be referred to as the `hard edge process'. For further discussion of the Stochastic Bessel operator see Ram\'irez and Rider \cite{RR}.

The square root of the hard edge process gives a point process description for the singular values of $\mathfrak{G}_{\beta,a}$. This scale is natural one for studying the transition from the edge to the bulk, and moreover, in this setting the asymptotic likelihood of rare events mimic those of the $\Sineb$ process. We will denote the singular value process by $\Bess$ in honor of the Bessel functions present in the determinantal description. The results for the $\Bess$ process will be stated in terms of its counting function $M_{a,\beta}(\lambda)$ which we define to be the number of points of the $\Bess$ process in the interval $[0,\lambda]$.

For a bit of amplification on the choice of the singular value process consider the following: we may perform the change of variables $y= \sqrt{x}$ in the Marchenko-Pastur distribution, the resulting distribution shows that the mean spacing after the change is the same order for both the edge and the bulk. This is confirmed by the work Edelman and LaCroix \cite{AEMLC}. They show that the singular values of a GUE are distributed as the union of the singular values of two independent Laguerre ensembles with hard edge type distribution.  A similar decomposition may also be done for the GOE \cite{FBMLC}.

In the bulk of the spectrum, with the appropriate centering and rescaling, Jacquot and Valk\'o showed that the eigenvalues of the $\beta$-Laguerre ensemble converge to the $\Sineb$ process \cite{SJBV}. This process was first introduced introduced as the limit of the $\beta$-Hermite ensemble by Valk\'o and Vir\'ag \cite{BVBV}. In this paper we make use of tools developed for the $\Sineb$ process to study the $\Bess$ process and show a transition from $\Bess$ to $\Sineb$. The $\Sineb$ process may be described via its counting function in the following way: let $\alpha_\lambda$ be a one parameter family of diffusions indexed by $\lambda$ that satisfy
\begin{equation}
\label{eq:alpha}
d\alpha_\lambda = \lambda \frac{\beta}{4}e^{-\frac{\beta}{4}t} dt + \re \left[ (e^{-i \alpha_\lambda}-1 )dZ\right],
\end{equation}
where $Z_t = X_t + i Y_t$ with $X$ and $Y$ standard Brownian motions and $\alpha_\lambda(0) = 0$. The $\alpha_\lambda$ are coupled through the noise term. Define $N_\beta(\lambda) = \frac{1}{2\pi}\lim_{t\to \infty} \alpha_\lambda(t)$, then $N_\beta(\lambda)$ is the counting function for $\Sineb$.

We might expect that as we move away from the edge of the Bessel process the effects of the edge will lessen and it will begin to behave like the bulk process. In this paper we show that this is indeed the case; there is a transition from $\Bess$ to $\Sineb$ as we move from the edge (near 0) out towards $\infty$ in the $\Bess$ process for $a>0$.   We also show two results on the asymptotic probability of various rare events for $\Bess$. The first is a central limit theorem for the number of points in the interval $[0,\lambda]$ as $\lambda \to \infty$. The second is a large deviation result on the asymptotic density of points in a large interval $[0,\lambda]$. We expect to see roughly $2\lambda / \pi$ many point in a large interval. We consider the asymptotic probability of seeing roughly $\rho \lambda$ many points for $\rho \ne 2/\pi$.

\subsection{Results}

We begin with the transition between the hard edge process and $\Sineb$. 
\begin{theorem}
\label{thm:transition}
Let $a>0$ and $\beta>0$ fixed, then
\begin{equation}
\frac{1}{4} (\Bess-\lambda) \Rightarrow \Sineb
\end{equation}
as $\lambda \to \infty$.
\end{theorem}
This can be understood by thinking of this as the distribution of the points in any neighborhood of $\lambda$ scaled down by 4 converges as $\lambda \to \infty$ to the distribution of $\Sineb$ in a neighborhood of 0. The centering of the neighborhood in the $\Sineb$ process is irrelevant since the process is translation invariant.

We now give the two results on the asymptotic behavior of $M_{a,\beta}(\lambda)$ as $\lambda \to \infty$. The first of these gives a central limit theorem for the number of points in the interval.

\begin{theorem}
\label{thm:bessclt}
Fix $\beta>0, a>-1$. As $\lambda \to \infty$ we have that 
\[
\frac{1}{\sqrt{\log \lambda} }\left( M_{a,\beta}(\lambda) - \frac{2 \lambda}{\pi} \right) \Rightarrow \mathcal{N}\Big( 0, \frac{1}{\beta \pi^2}\Big).
\]
\end{theorem}
A similar result except with limiting variance $\frac{2}{\beta \pi^2}$ was shown using a different method for the counting function of the $\Sineb$ process by Kritchevski, Valk\'o, and Vir\'ag \cite{KVV}.

The next result describes the large deviation behavior of the counting function. Before stating the result we introduce certain special functions that are used in the statement. We will use
\begin{align}
K(m)= \int_0^{\pi/2} \frac{dx}{\sqrt{1-m \sin^2 x}}, \quad \text{ and } \quad
E(m)= \int_0^{\pi/2} \sqrt{1-m\sin^2 x}\ dx, \label{KandE}
\end{align}
for the complete elliptic integrals of the first and second kind, respectively. Note that there are several conventions denoting these functions, we use the modulus notation from \cite{AbSt}.
 We also introduce the following  function for $m<1$:
\begin{align}
\label{defH}
\II(m)&= (1-m)K(m)-E(m).
\end{align}
\begin{theorem}
\label{thm:bess}
Fix $\beta>0$, $a>-1$.  The sequence of random variables $\frac{1}{ \lambda}M_{a,\beta}(\lambda)$ satisfies a large deviation principle with scale $\lambda^2$ and good rate function $\beta \bessldp (\rho)$ with
\begin{align}\label{defsineldp}
\bessldp (\rho)=  \frac{\nu}{2}+ \rho \II(\nu),\qquad \nu=\gamma^{(-1)}(\rho/4),
\end{align}
where $\gamma^{(-1)}$denotes the inverse of the continuous, strictly decreasing function given by 
\begin{align}
\gamma(\nu)= \begin{cases} \,\,\frac{\II(\nu)}{8} \int\limits_{-\infty}^\nu \II^{-2}(x)dx, & \textup{ if }\nu<0,\\[14pt]
\qquad\quad  \tfrac{1}{2\pi}, & \textup{ if } \nu=0,\\[14pt]
\,\, \frac{\II(\nu)}{8} \int\limits_{1}^\nu \II^{-2}(x)dx, & \textup{ if } 0< \nu< 1,\\[14pt]
\qquad\quad 0, &\textup{ if } \nu=1.
 \end{cases}\label{rhonu}
\end{align}
\end{theorem}
\noindent Roughly speaking, this means that the probability of seeing close to  $\rho  \lambda$ points in $[0,\lambda]$ for a large $\lambda$ is asymptotically $e^{-\lambda^2 \beta \bessldp (\rho)}$.

This result is closely related to the analogous result for the counting function $N_\beta(\lambda)$ of the $\Sineb$ process. There we consider the sequence $\frac{1}{\lambda} N_\beta(\lambda)$ and we find an LDP with rate $\lambda^2$ and rate function $\beta \sineldp$ where $\bessldp (\rho) = 32 \sineldp (\rho/4)$\cite{DHBV}.  Moreover we can check that the central limit theorem and the large deviation result are at least formally consistent. Lastly, observe that the large deviation result is also consistent with the tail behavior of the lowest eigenvalue $P(M_{a,\beta}(\sqrt \lambda) = 0) \sim e^{-\tfrac{\beta}{2} \lambda}$ \cite{RRZ}.

For a bit of clarification on why we see consistency between results on the bulk and hard-edge processes we introduce the following characterization of $\Bess$.

\begin{theorem}
\label{M}
Let $M_{a, \beta}(\lambda)$ be the number of points of $\Bess$ in the interval $[0,\lambda]$, and let $\varphi_{a,\lambda}$ be the diffusion that satisfies the stochastic differential equation
\begin{equation}
\label{phi}
d\varphi_{a,\lambda} = \frac{\beta}{2} (a+\tfrac{1}{2}) \sin \big( \frac{\varphi_{a,\lambda}}{2}\big) dt + \beta  \lambda e^{-\beta t/8} dt + \frac{\sin \varphi_{a,\lambda}}{2} dt + 2 \sin \big(\tfrac{\varphi_{a,\lambda}}{2}\big) dB_t
\end{equation}
with initial condition $\varphi_{a,\lambda}(0)=2\pi$.  Then
\[
 M_{a,\beta}(\lambda) =^d  \lim_{t\to \infty}   \left\lfloor  \frac{1}{4\pi}   \varphi_{a,\lambda}(t) \right\rfloor .
\]
Moreover, for $a>0$ we have that $\lim_{t\to \infty}  \lfloor(\varphi_{a,\lambda}(t)-2\pi)/4\pi \rfloor  = \lim_{t\to \infty}  \lfloor \varphi_{a,\lambda}(t)/4\pi \rfloor $ almost surely. 
\end{theorem}
We also make the observation that for a fixed $\lambda$ the $\alpha_\lambda$ diffusion used in the characterization of the bulk process satisfies the SDE
\[
d\alpha_\lambda = \lambda \frac{\beta}{4} e^{-\frac{\beta}{4}t} dt + 2 \sin \left( \frac{\alpha_\lambda}{2} \right) dB_t, \qquad \alpha_\lambda(0)=0,
\]
where $B_t$ is a standard Brownian motion. The Brownian motion $B_t$ that appears depends on the $\lambda$ parameter.

Notice now that for $\lambda$ large the $\varphi_{a,\lambda}$ diffusion will be rapidly increasing until time on the order of $\log \lambda$. On this region the finite variation terms involving $\sin ( \frac{\varphi_{a,\lambda}}{2})$ and $\sin \varphi_{a,\lambda}$ will be rapidly oscillating and so have a minimal contribution. Essentially these terms are not felt in the $\lambda \to \infty$ limit and so they vanish in asymptotic results. The results on oscillatory integrals involving  $\varphi_{a,\lambda}$ will turn out to be the key component in the proof of all 3 main results and will be given in section 2. 

It is worth noting that from the characterization in Theorem \ref{M} it seems likely that one could show other results for $\Bess$ related to existing results on the $\Sineb$ process. In particular we anticipate it would not be difficult to determine the asymptotic probability of overcrowding ($P(M_{a,\beta}(\lambda) \geq n) \sim ?$ as $n\to \infty$, see \cite{DHBV2}).

The remainder of the paper will be organized as follows: Section 2 will give the proof of Theorem \ref{M} as well as several results on the $\varphi_{a,\lambda}$ diffusion. Section 3 will give the proof of the transition from $\Bess$ to $\Sineb$. Section 4 will give the proof of the central limit theorem.  Section 5 will give the proof of the large deviation result.

\bigskip

\noindent{\bf Acknowledgements:} The author would like to thank Benedek Valk\'o for helpful comments and corrections.

\section{The counting function of $\Bess$}

Before giving the proof of Theorem \ref{M} we recall an existing description of $\Bess$ that characterizes the process via diffusions rather then an operator.  We consider the `Riccati diffusion' for $\mathfrak{G}_{\beta,a}$, given by the stochastic differential equation
\begin{equation}
d p_\lambda(t) = \frac{2}{\sqrt \beta} p_\lambda(t) dB(t) + \left( \left( a + \tfrac{2}{ \beta}\right) p_\lambda(t) - p_\lambda^2(t) - \lambda e^{-t} \right) dt,
\end{equation}
with initial condition $p(0)=+\infty$, which it leaves instantaneously.  Note that there is a positive probability of explosion to $-\infty$.

\begin{theorem}[\cite{RR}]
\label{SDEchar}
Let $\Lambda_0(\beta,a)< \Lambda_1(\beta,a)<...$ be the ordered eigenvalues of $\mathfrak{G}_{\beta,a}$, and let $P_{\infty, t}$ denote the law induced by $p(\cdot: \beta, a, \lambda)$ started at $+\infty$ at time $t$, and restarted at $+\infty$ and time $\mathfrak{m}$ upon any $\mathfrak{m}< \infty, p(\mathfrak{m}) = -\infty$.  Then,
\begin{align}
P(\Lambda_0(\beta,a)>\lambda) & = P_{\infty, 0}(p \text{ never hits } 0 ),\\
P(\Lambda_k(\beta,a)< \lambda) & = P_{\infty,0}(p \text{ hits 0 at least } k+1 \text{ times}).
\end{align}

\end{theorem}
\noindent In other words the counting function of the process is the number of times that $p_\lambda(t)$ hits 0, and may be denoted by $M_{a,\beta}(\sqrt \lambda)$, where $M_{a,\beta}$ is the counting function of $\Bess$.

\begin{proof}[Proof of Theorem \ref{M}]

The characterization of the hard edge process given in Theorem \ref{SDEchar} can be rewritten in the following way: 
For $p_\lambda>0$ apply the change of variables $-X_1(t) := \log (p_\lambda(\beta t/4))+ \beta t / 8 - \log \lambda / 2$ for $p_\lambda>0$.  Then $X_1$ satisfies the SDE
\begin{equation}
\label{eq:x1}
d X_1(t) = \left( \frac{\beta}{4} (-a- \tfrac{1}{2}) + \frac{\beta}{2} \sqrt \lambda e^{-\beta t/ 8} \cosh X(t)\right) dt - dB(t).
\end{equation}
The initial condition $p(0) = +\infty$ gives $X_1(0)= -\infty$, moreover when $p_\lambda(\beta t/4)$ reaches $0$ we get that $X_1(t) = +\infty$.

We do a similar change of variables for $p_\lambda<0$. Take $X_2(t) =  \log (-p_\lambda(\beta t/4))+ \beta t/8 - \log \lambda/2$.  This gives us 
\begin{equation}
\label{eq:x2}
d X_2(t) = \left( \frac{\beta}{4} (a +\tfrac{1}{2}) + \frac{\beta}{2} \sqrt \lambda e^{-\beta t/ 8} \cosh X_2(t)\right) dt + dB(t).
\end{equation}
The boundary condition $p_\lambda(\beta t/4)=0$ gives $X_2(t) = -\infty$, and for $p_\lambda(\beta t/4) = - \infty$ we get $X_2(t) = + \infty$.

To find the $\varphi_{a,\lambda}$ diffusion given in Theorem \ref{M} we work back from $X_1$ and $X_2$ to $\varphi_{a,\sqrt{\lambda}}$. Notice that the zeros of $p_\lambda$ describe the eigenvalue process of $\mathfrak{G}_{\beta,a}$ and so the resulting diffusion have parameter $\sqrt{\lambda}$. Suppressing the subscripts $a,\sqrt \lambda$, let $\varphi = -4 \arctan e^{-X_1}$, then we get
\begin{align*}
d \varphi &= 2 \sech X_1 d X_1 - \sech X_1 \tanh X_1 dt\\
& =   \frac{\beta}{2}(a+\tfrac{1}{2}) \sin \big( \frac{\varphi}{2} \big)  dt +  \beta \sqrt \lambda e^{-\beta t/ 8} dt   +\frac{ \sin \varphi }{2} dt  +  2 \sin \big( \frac{\varphi}{2} \big) dB_t.
\end{align*}
The conditions $X_1 = -\infty$ and $X_1 = +\infty$ become $\varphi =-2\pi$ and $\varphi =0$ respectively.  For $X_2$ take $\varphi = 4 \arctan e^{X_2}$.  This gives 
\begin{align*}
d\varphi & = \frac{\beta}{2} (a+\tfrac{1}{2}) \sin \big( \frac{\varphi}{2}\big) dt + \beta \sqrt \lambda e^{-\beta t/8} dt + \frac{\sin \varphi}{2} dt + 2 \sin \big(\frac{\varphi}{2}\big) dB_t.
\end{align*}
The conditions $X_2 = -\infty$ and $X_2 = + \infty$ become $\varphi =0$ and $2\pi$ respectively.

Now notice that this diffusion is invariant under $4\pi$ spacial shifts, so for a fixed $\lambda$ with initial condition $\varphi_{a,\sqrt{\lambda}}(0)=2\pi$ we have
\[
P(\sup_t \varphi_{a,\sqrt \lambda}(t) \geq 4\pi k) = P(\Lambda_k < \lambda).
\]
 Lastly we use that $\lfloor\varphi_{a,\sqrt \lambda}\rfloor_{4\pi}$ is monotone nondecreasing in $t$ to rewrite the supremum as a limit. 

For the final statement when $a>0$ we appeal to the erratum for the original convergence result on the hard edge, \cite{RRerr}. Observe that in this case counting $0$ of the $p_\lambda$ diffusion is equivalent to counting explosions to $-\infty$. Therefore any time $\varphi_{a,\lambda}$ passes a multiple of $4\pi$ is must pass the next $2\pi$ multiple as well.
\end{proof}

It will be useful to consider what is the relationship between two diffusions that satisfy stochastic differential equations of the form (\ref{phi}) with two different $\lambda$s which are coupled through their noise terms. Let $\psi_{a,\lambda,x}= \varphi_{a,\lambda+x}- \varphi_{a,\lambda}$, then the SDE for $\psi_{a,\lambda,x}$ is
\begin{align}
d \psi_{a,\lambda,x} &= \frac{\beta}{2}(a+1/2)\im \left[e^{ i\frac{ \varphi_{a,\lambda}}{2}}\left(e^{- i\frac{\psi_{a,\lambda,x}}{2}}-1\right)\right]dt  + \frac{1}{2}\im \left[e^{ i\varphi_{a,\lambda}}\left(e^{- i\psi_{a,\lambda,x}}-1\right)\right]dt \notag\\
&\hspace{1.5cm} + \beta x e^{-\beta t/8}dt
+\im \left[e^{ i\frac{ \varphi_{a,\lambda}}{2}}\left(e^{- i\frac{\psi_{a,\lambda,x}}{2}}-1\right)\right]dB_t,
\label{psilambda}
\end{align} 
with initial condition $\psi_{a,\lambda,x}(0)=0$. This follows from standard It\^o techniques together with the application of angle addition formulas. This rather ugly formula can be made more palatable by the observation that the oscillatory terms may be well controlled.

\begin{proposition}
\label{prop:oscillation}
Let $\varphi_{a,\lambda}$ and $\psi_{a,\lambda,x}$ be defines as above, then for $T\leq \frac{8}{\beta} \log \lambda$ there exists a constants $M$ and $\gamma$ (uniform in $\lambda$ and $T$) such that 
\begin{align}
&E \left|\sup_{0\le s \le T}\int_0^{s} e^{i c\varphi_{a,\lambda}} dt \right| \le \frac{ M }{\lambda \mathfrak{h}(T)}, \quad \text{ and } \label{eq:oscillation1}\\
& E \left|\sup_{0\le s \le T}\int_0^{s}e^{i c\varphi_{a,\lambda}} \left( e^{-i c\psi_{a,\lambda,x}}-1\right)dt\right| \le \frac{ M }{\lambda \mathfrak{h}(T)}. \label{eq:oscillation2}
\end{align}
In particular in the case where $T$ is fixed this gives $\sup_{0\le s \le T}\int_0^{s} \sin\left(c\varphi_{a,\lambda}\right) dt \to 0$ in $L_1$ (and hence in probability) as $\lambda \to \infty$ (and similarly for $\cos(c \varphi_{a,\lambda})$ and integrals related to (\ref{eq:oscillation2})).
Moreover
\begin{align}
P \left( \sup_{0 \le s \le T} \int_0^T e^{i c\varphi_{a,\lambda}}dt - \frac{M}{\lambda \mathfrak{h}(T)} \ge C\right) \le \exp \left[ - C^2 \lambda^2 \gamma e^{-\frac{\beta}{4}T} \right].
\label{eq:oscillation3}
\end{align}
A similar bound on the tail of the integral appearing in (\ref{eq:oscillation2}) also holds.
\end{proposition}

\begin{proof}[Proof of Proposition \ref{prop:oscillation}]
We write $\varphi_{a,\lambda}(t)$ in its integrated form (dropping the subscripts)
\[
\varphi(t)  = \lambda 8\left[1-e^{-\frac{\beta}{8}t}\right] + 2\pi +\frac{\beta}{2} (a+1/2) \int_0^t \sin \big( \frac{\varphi}{2}\big) ds +\int_0^t \frac{\sin (\varphi)}{2} ds + 2 \int_0^t\sin \big(\tfrac{\varphi}{2}\big) dB_s.
\]
We break this into two pieces, the first term we will write $\lambda H(t)=\lambda 8\left[1-e^{-\frac{\beta}{8}t}\right]$ and the remaining terms will be grouped together as the process $\mathcal{E}_t= \varphi(t)-\lambda H(t)$.
Then
\begin{align*}
\int_0^T e^{i c \varphi_{a,\lambda}}dt &= \int_0^T e^{i c \lambda H(t)} e^{i c \mathcal{E}_t}dt.
 \end{align*}
 We use the following version of It\^o's formula to extract the main term. Let $f,g$ be continuously differentiable functions and let $G$ denote the antiderivative of $g$. Then for $X$ and It\'o process we have 
 \[
 \int_0^T f'(t)G(X)dt = f(T)G(X)-f(0)G(0)- \int_0^T f(t)g(X)dX- \frac{1}{2} \int_0^T f(t)g'(X)d[X]_t.
 \]
Let $\Lambda(t) = \int_0^t e^{ic \lambda H(t)} dt$, then 
 \begin{align*}
 \int_0^T e^{i c \lambda H(t)} e^{i c \mathcal{E}_t} dt& = e^{i c \mathcal{E}_t}\Lambda(T)+ \int_0^T \Lambda(t) i c e^{i c \mathcal{E}_t}d\mathcal{E}_t - \frac{1}{2}\int_0^T \Lambda(t)c^2 e^{i c \mathcal{E}_t} d[\mathcal{E}]_t.
 \end{align*}
 Now observe that $\Lambda(t)$ may be bounded in the following way:
\begin{align*}
\int_0^{t} e^{i c\lambda  H(s)}ds  &= \int_0^{t} \frac{1}{\lambda i c H'(s)} \frac{d}{ds} e^{ic \lambda  H(s)}ds= \frac{ e^{\frac{\beta}{8}t}}{\lambda i c \beta}e^{i c\lambda  H(t)} - \frac{1}{8\lambda c } \int_0^{t} e^{\frac{\beta}{8}s}  e^{i c\lambda  H(s)}ds.
\end{align*}
In absolute value the final integral is bounded by $\frac{1}{\lambda c \beta}(e^{\frac{\beta}{8}}t-1)$ which gives us that 
\begin{equation}
\label{eq:GammaBound}
|\Lambda(t)| = \left|\int_0^{t} e^{ic \lambda  H(s)}ds\right| \le \frac{2}{\lambda c \beta} e^{\frac{\beta}{8}t}.
\end{equation}
From this we get that the $dt$ terms in the $d \mathcal{E}_t$ integral may be bounded in absolute value by $[\frac{\beta}{2}(a+1/2)+ 1/2]\int_0^T \Lambda(t) dt  \le [\frac{\beta}{2}(a+1/2)+ 1/2]\frac{16}{\lambda c \beta^2} e^{\frac{\beta}{8}T}$. Similar computation shows that the $d[\mathcal{E}]_t$ term is bounded in absolute value by $\frac{64}{\lambda c \beta^2} e^{\frac{\beta}{8}T}$. Lastly, for the martingale term we break it into its real and complex parts and use Doob's martingale inequality on the associated exponential submartingales. We show the argument for the imaginary part. The complex part may be done the same way. Let
\begin{align}
N_t &= 2 \int_0^t \Lambda(s) \sin\big( c \mathcal{E}_s\big) \sin \big(\tfrac{\mathcal{E}_s + \lambda H(s)}{2}\big)dB_s,
\end{align}
then $N_t$ is a true martingale because it has $L^1$ bounded quadratic variation. Therefore $\exp( \xi N_t)$ is a positive submartingale and so $P( \sup_{0 \le t \le T} \exp( \xi N_t) \ge x) \le E(\exp( \xi N_T))/x$. From this we get that
\begin{equation}
P\left( \sup_{0 \le t \le T} N_t \ge C\right) \le e^{-\xi C} E\left(\exp( \xi N_T)\right) 
\end{equation}
To compute $E \exp( \xi N_t)$ we make use of the martingale $M_t = \exp ( \xi N_t - \frac{\xi^2}{2}[N]_t)$. This gives us that 
\begin{align}
1 = EM_t  \ge E\left[\exp \left( \xi N_t - \frac{\xi^2}{2}\frac{1}{\beta}(\frac{8}{\lambda c \beta}e^{\frac{\beta}{8} t})^2\right)\right] = E(\exp(\xi N_t)) \exp \left[- \frac{\xi^2}{2}\frac{1}{\beta}(\frac{8}{\lambda c \beta}e^{\frac{\beta}{8} t})^2\right],
\end{align}
where we use
\begin{equation}
[N]_t \le  \int_0^t 4 \Lambda^2(t)dt \le \frac{1}{\beta}(\frac{8}{\lambda c \beta}e^{\frac{\beta}{8} t})^2.
\end{equation}
Optimizing our choice of $\xi$ we get 
\begin{align}
P\left( \sup_{0 \le t \le T} N_t \ge C\right) \le \exp \left[- \frac{C^2}{2} \beta \left( \frac{\lambda c \beta}{8}e^{-\frac{\beta}{8}T}\right)^2  \right].
\end{align}
This gives the necessary bound for (\ref{eq:oscillation3}). Integrating in the $C$ variable gives
\begin{align}
E(\sup_{0 \le t \le T} N_t ) \le \sqrt{ \frac{\pi}{2\beta }} \frac{8}{\lambda c \beta}e^{\frac{\beta}{8} T},
\end{align}
 which completes the proof of equation (\ref{eq:oscillation1}). 

To extend this to the case where we have the additional $ e^{i c\psi_{\lambda,f,g}}-1$ multiplier in the integral we use the same decomposition of $\xi_{f,a,\lambda,0}$ and work with the integral. 
\[
\int_0^{T}\sin(c \lambda H(t)) \cos (c \mathcal{E}_t) \left( \cos(c\psi_{\lambda,f,g})-1\right)dt.
\]
An application of a slightly modified It\'o's lemma leads us to the same type of analysis as before and give the necessary bounds for (\ref{eq:oscillation2}). In particular for $u,v,w$ continuously differentiable functions with $V,W$ the antiderivatives of $v$ and $w$ we get 
\begin{align*}
\int_0^T u'(t) V(Z)W(Y) dt & = u(T)V(Z_T)W(Y_T)- u(0)V(Z_0)W(Y_0) - \int_0^T u(t)v(Z)W(Y)dZ \\
& \qquad - \int_0^T u(t)V(Z)w(Y)dY - \frac{1}{2} \int_0^T u(t)v'(Z)W(Y)d[Z] \\
& \qquad - \frac{1}{2} u(t)V(Z)w'(Y)d[Y]- \int_0^T u(t)v(Z)w(Y)d[Z,Y].
\end{align*}
All of the finite variation terms the integrand may be bounded in absolute value by $\tilde M\Lambda(t)$ for some constant $\tilde M$. The martingale terms may be handled the same way as before.
\end{proof}

\section{From the hard edge to the bulk}
\label{sec:transition}

In order to show the transition we need the following two results on limits of martingales and stochastic integrals:
\begin{proposition}[Multidimensional Martingale CLT]
\label{prop:martingaleCLT}
Let $\{ M_n(\cdot)\}$ be a sequence of $\rr^d$ valued martingales. Suppose 
\[
\lim_{n\to \infty} E\big[ \sup_{s\leq t} |M_n(s)-M_n(s-)|\big] = 0 
\]
and $[M_n^i, M_n^j]_t \to c_{i,j}(t)$ in probability for all $t\geq 0$ where $C= [c_{i,j}]$ is a continuous, symmetric matrix valued function on $[0,\infty)$ with $C(0)=0$ and 
\[
\sum_{i,j \le n} (c_{ij}(s)-c_{ij}(t)) \xi_i\xi_j>0, \quad \text{ for } \xi \in \rr^n, \quad t>s\ge 0.
\] 
Then $M_n \to M$, where $M$ is a Gaussian process with independent increments and $E[M(t)M(t)^T] = C(t)$.
\end{proposition}
For a proof see e.g. Theorem 7.1.4 in Ethier and Kurtz \cite{EthierKurtz}.

The following proposition gives conditions under which a sequence of diffusions $X_n$ satisfying the stochastic integral equations
\begin{align}
\label{eq:diffusionseq}
X_n(t) = X_n(0) + \int_0^t \sigma(X_n,s-) dM_n(s) + \int_0^t b(X_n,t) dV_n(t)
\end{align}
converge to a limiting process. Here we take $M_n: \rr_+ \to (C[0,\infty))^d$ to be a $d-$dimensional martingale and $V_n:\rr_+ \to \rr^{d\times d}$ is a finite variation process. The following is a specialization of Theorem 5.4 from from \cite{KP}.

\begin{proposition}[Kurtz, Protter \cite{KP}]
\label{prop:SDEconvergence}
Suppose in (\ref{eq:diffusionseq}) $M_n$ is a martingale, and $V_n$ is a finite variation process. Assume that for each $t\geq 0$, $\sup_n E[[M_n]_t]<\infty$ and $\sup_n E[TV(V_n)]<\infty$ (where $TV$ indicates the total variation) and $(M_n, V_n) \Rightarrow (W,V)$, where $W$ is a standard Brownian motion and $V(t)=tI$. Suppose that the diffusion $X$ satisfies
\begin{equation}
\label{eq:diffusionlimit}
X(t) = X(0) + \int_0^t \sigma(X(s),s) dW(t) + \int_0^t b(X(s),s)ds
\end{equation}
and that (\ref{eq:diffusionlimit}) has a unique strong solution. Then $X_n \Rightarrow X$.
\end{proposition}

Lastly we will need a property of the diffusion $\alpha_\lambda(t)$. 
\begin{proposition}[\cite{BVBV}]
\label{prop:alphalimit}
The diffusion $\alpha_\lambda$ satisfying (\ref{eq:alpha}) converges a.s. to a multiple of $2\pi$ as $t \to \infty$.
\end{proposition}

\begin{proof}[Proof of Theorem \ref{thm:transition}]
We study the difference $M_{a,\beta}(\lambda+x)- M_{a,\beta}(\lambda)$ as $\lambda \to \infty$ by way of the SDE characterization. For convenience we use the diffusion $\psi_{a,\lambda, x} = \varphi_{a,\lambda+x}-  \varphi_{a,\lambda}$ defined in section 2 (see (\ref{psilambda})). The proof breaks into two pieces. The first is to show that for any finite set $\{x_1, x_2,...,x_k\}$ the family of diffusions $\{\psi_{a,\lambda,x_i}\}_{i=1,...,k}$ converge weakly to a family of diffusions $\{\hat \psi_{x_i}\}_{i=1,...,k}$ on compact sets $[0,T]$. We then show that for $T$ sufficiently large this convergence is enough, that is that the tail of the diffusion is well behaved. These together will be sufficient for process convergence because they will show that the finite marginals of $M_{a,\beta}(\lambda+x)- M_{a,\beta}(\lambda)$ will converge to those of $N_\beta(x/4)$ where $N_\beta$ is distributed as the counting function of $\Sineb$.

Recall that $\psi_{a,\lambda,x}$ satisfies the SDE in (\ref{psilambda}). In order to study the limit as $\lambda \to \infty$ we start by showing that the martingales
\begin{equation}
W_{\lambda,1}(t) =\sqrt{2} \int_0^t \sin\left(\frac{\varphi_{a,\lambda}}{2}\right)dB_t \qquad \text{ and } \qquad W_{\lambda,2}(t) =\sqrt{2} \int_0^t \cos\left(\frac{ \varphi_{a,\lambda}}{2}\right)dB_t
\end{equation}
converge in distribution to two independent Brownian motions $W_1(t)$ and $W_2(t)$ as $\lambda \to \infty$. We then use these limits to show the convergence of the vector $(\psi_{a,\lambda,x_1}(t),...,\psi_{a,\lambda,x_k}(t))$ to $(\hat \psi_{x_1}(t),\dots, \hat \psi_{x_k}(t))$, which is a vector of time and space changed versions of $\alpha_x$, for $x=x_1,...,x_k$.

We show that $(W_1, W_2)$ is a 2 dimensional Brownian using Proposition \ref{prop:martingaleCLT}. Since $W_{\lambda,1}(t)$ and $W_{\lambda,2}(t)$ are continuous we need only check the quadratic variations $[W_{\lambda,1}]_t, [W_{\lambda,2}]_t\to t$ and $[W_{\lambda,1},W_{\lambda,2}]_t \to 0$ as $\lambda \to \infty$. We use the integral representation of $W_{\lambda,1}$  and Proposition \ref{prop:oscillation} to observe that for any $T\in [0,\infty)$
\[
[W_{\lambda,1}]_t =2 \int_0^t \sin^2 \left( \frac{\varphi_{a,\lambda}}{2}\right) dt=t- \int_0^{t} \cos(\varphi_{a,\lambda})  ds  \to t
\]
in probability as $\lambda \to \infty$ for $t\in [0,T]$. Similar calculations may be done for $[W_{\lambda,2}]_t$ and $[W_{\lambda,1},W_{\lambda,2}]_t$.

To show the SDE convergence we can again use Proposition \ref{prop:oscillation} to show that the first two drift terms in (\ref{psilambda}) vanish as $\lambda \to \infty$. We have now identified that for a fixed $x$ the limiting diffusion should satisfy the SDE
\begin{align}
\label{eq:psi}
d\hat \psi_{x} =  \beta x e^{-\beta t/8}dt +\frac{1}{\sqrt{2}} \left[\cos \left(\frac{\hat \psi_x}{2}\right)-1\right]dW_1 +\frac{1}{\sqrt{2}} \sin \left(\frac{\hat \psi_x}{2}\right)dW_2, \qquad \psi_{x}(0)=0.
\end{align}
Proposition \ref{prop:SDEconvergence} gives us that if this SDE has a strong solution, then $\hat \psi_{x} \ed \lim_{\lambda \to \infty} \psi_{a,\lambda,x}$ exists and satisfies (\ref{eq:psi}). To show that (\ref{eq:psi}) has a strong solution we will show this is equivalent to the SDE for $\alpha_\lambda$ having a strong solution.  Apply time change $t/2 \mapsto s$ and space change $\hat \psi_x/2 \mapsto \tilde \psi_x$ to get 
\begin{equation}
d \tilde \psi_x = \beta x e^{-\beta s/4}ds + \textup{Re} \left[(e^{-i\tilde \psi_x}-1)d(\tilde W_1+i\tilde W_2)\right],
\end{equation}
where $\tilde W_1$ and $\tilde W_2$ here are the time changed $W_1$ and $W_2$.  Notice that the time change is independent of $x$ and so $\tilde W_1$ and $\tilde W_2$ are the same driving Brownian motions for $x_1,...,x_k$. Therefore the diffusions $\tilde \psi_{x_i}$ are coupled through their Brownian terms. Lastly observe that $\tilde \psi_{x_i} \ed \alpha_{4x_i}$ and $\alpha_\lambda$ has a strong solution \cite{BVBV}.   

For the next step we need to show that convergence of the diffusions implies convergence of the processes $\Bess$ to $\Sineb$. To do this we show convergence of the finite dimensional marginals of the counting function. That is we show that for any finite collection $\{x_1,...,x_k\}\in \rr$ we have 
\[
\{M_{a,\beta}(\lambda + x_i)-M_{a,\beta}(\lambda) \}_{i=1,...,k} \Rightarrow \{N_{\beta}(x_i/4)\}_{i=1,...,k}
\]
jointly in law as $\lambda \to \infty$. 

We begin with the following lemma, which gives us that it is enough to study the diffusion $\psi_{a,\lambda,x}$.
\begin{lemma}
\label{lem:psilimit}
For $\psi_{a,\lambda,x}$ and $M_{a,\beta}(\lambda)$ defines as above we have
\[
P\left(\lim_{t\to \infty} \lfloor (\psi_{a,\lambda,x}(t)+2\pi)/4\pi\rfloor = M_{a,\beta}(\lambda+x)- M_{a,\beta}(\lambda)\right)=1.
\]
\end{lemma}
For convenience we introduce the notation $\lfloor y \rfloor_{4\pi} = 4\pi \lfloor y/(4\pi)\rfloor$. We now show that for any $\eps>0$ we can choose $\lambda$ and $T$ sufficiently large so that 
\begin{align}
P&\left(\lfloor\psi_{a,\lambda,x_i}(\infty)+2\pi\rfloor_{4\pi}= \lfloor \psi_{a,\lambda,x_i}(T)+2\pi\rfloor_{4\pi}, i = 1,\dots,k \right) > 1- \eps, \quad \text{ and } \label{eq:psilimit} \\
P&\left( \hat \psi_{x_i}(\infty) = \left\lfloor \hat \psi_{x_i}(T)+2\pi \right\rfloor_{4\pi}, i = 1,\dots, k\right) > 1-\eps, \label{eq:hatpsilimit}
\end{align}
where evaluation at $\infty$ should be understood as a limit. Since we have that 
\[
\{\psi_{a,\lambda,x_1}(T), \dots, \psi_{a,\lambda,x_k}(T)\} \Rightarrow \{\hat \psi_{x_1}(T), \dots, \hat \psi_{x_k}(T)\} \quad \text{ as } \quad \lambda \to \infty
\]
this will be sufficient to complete the proof. 

For equation (\ref{eq:hatpsilimit}) Proposition \ref{prop:alphalimit} gives us that $\hat \psi_x$ converges almost surely to a multiple of $4\pi$. Therefore for a finite collection $\{x_1,\dots, x_k\}$ and $0<\delta<2\pi$ we may choose $T$ sufficiently large so that $P(|\hat \psi_{x_i}(T)- \hat\psi_{x_i}(\infty)|<\delta, i= 1, \dots, k)>1-\eps$.  From this (\ref{eq:hatpsilimit}) follows.

For equation (\ref{eq:psilimit}) we begin with the bound
\begin{align}
P&\left(\left\lfloor \psi_{a,\lambda,x_i}(\infty)+2\pi\right\rfloor_{4\pi}= \left\lfloor \psi_{a,\lambda,x_i}(T)+2\pi \right\rfloor_{4\pi}, i = 1,\dots,k \right) \\
&\qquad \qquad \ge 1 - \sum_{i=1}^k P\left(\left\lfloor \psi_{a,\lambda,x_i}(\infty)+2\pi\right\rfloor_{4\pi}\ne \left\lfloor \psi_{a,\lambda,x_i}(T)+2\pi \right\rfloor_{4\pi} \right) 
\end{align}
And observe that
\begin{align}
P&\left(\left\lfloor \psi_{a,\lambda,x_i}(\infty)+2\pi\right\rfloor_{4\pi}\ne \left\lfloor \psi_{a,\lambda,x_i}(T)+2\pi \right\rfloor_{4\pi} \right) \\
& \le 1- P\left(\left\lfloor \psi_{a,\lambda,x_i}(\infty)+2\pi\right\rfloor_{4\pi}= \left\lfloor \psi_{a,\lambda,x_i}(T)+2\pi \right\rfloor_{4\pi}, | \psi_{a,\lambda,x_i}(T)- \lfloor  \psi_{a,\lambda,x_i}(T)+2\pi\rfloor_{4\pi}|<\delta \right). \notag
\end{align}
We use the following lemma to complete the proof.
\begin{lemma}
\label{lemma:conditioned}
For $|\eta|<\delta < 1/16$ there exists some $A$ (uniform in $x$) so that for $T \ge - \frac{8}{\beta}\log (\delta)$
\begin{align}
P\Big(\left\lfloor \psi_{a,\lambda,x}(\infty)+2\pi\right\rfloor_{4\pi}\ne \left\lfloor \psi_{a,\lambda,x}(T)+2\pi \right\rfloor_{4\pi} \Big| \psi_{a,\lambda,x}(T)- \lfloor \psi_{a,\lambda,x}(T)+2\pi& \rfloor_{4\pi}= \eta \Big)\label{eq:conditioned} \\
 & \le (x+A)\sqrt{\delta} \notag. 
\end{align}
\end{lemma}
This gives us that 
\begin{align*}
P\Big(\left\lfloor \psi_{a,\lambda,x_i}(\infty)+2\pi\right\rfloor_{4\pi}\ne \left\lfloor \psi_{a,\lambda,x_i}(T)+2\pi \right\rfloor_{4\pi}, | \psi_{a,\lambda,x_i}(T)- \lfloor  \psi_{a,\lambda,x_i}(T)+&2\pi \rfloor_{4\pi}|<\delta \Big)\\
& \le (x_i+A)\sqrt{\delta}.\notag
\end{align*}
In order to complete the bound in equation (\ref{eq:psilimit}) we follow the reasoning from when we proved (\ref{eq:hatpsilimit}). We observe that $\psi_{a,\lambda,x}(T)$ converges in distribution to $\hat \psi(T)$ as $\lambda \to \infty$, and $\hat \psi_{x_i}(T)$ will be close to a multiple of $4\pi$ with high probability when $T$ is large enough. Therefore for any $\eps>0$ we may choose $T$ and $\lambda$ large enough so $P(| \psi_{a,\lambda,x_i}(T)- \lfloor  \psi_{a,\lambda,x_i}(T)+2\pi \rfloor_{4\pi}|<\delta)>1-\eps$. Together these bounds give us that 
\[
P\left(\left\lfloor \psi_{a,\lambda,x_i}(\infty)+2\pi\right\rfloor_{4\pi}\ne \left\lfloor \psi_{a,\lambda,x_i}(T)+2\pi \right\rfloor_{4\pi} \right) \le \eps + (x/4+A)\sqrt{\delta},
\]
where $\eps$ and $\delta$ may be chose arbitrarily small.
\end{proof}

\begin{proof}[Proof of Lemma \ref{lemma:conditioned}]

We first show this for the case where $\eta>0$, in which case $\lfloor \psi_{a,\lambda,x}(T)\rfloor_{4\pi} = \lfloor \psi_{a,\lambda,x}(T)+2\pi\rfloor_{4\pi}$. We study the diffusion $\psi_{a,\lambda,x}$ on two regions. The first is $[T,T_{\lambda,\delta}]$ and the second is $[T_{\lambda,\delta},\infty)$, where $T_{\lambda,\delta}= \frac{8}{\beta}\log (\lambda  \delta)$. For $x>0$ the process $\lfloor\psi_{a,\lambda,x}(t)\rfloor_{4\pi}$ is monotone increasing in $t$, therefore it follows that $\psi_{a,\lambda,x}(\infty)-\lfloor \psi_{a,\lambda,x}(T)\rfloor_{4\pi}$ is a strictly positive random variable so we can apply Markov's inequality. We use the integral form of (\ref{psilambda}) together with Proposition \ref{prop:oscillation} to get a bound on the expected value. Let $\hat M = \frac{\beta}{8}(\beta^2(a+1/2)+1) M$, then from the integrated form we have
\begin{align*}
E\big(\psi_{a,\lambda,x}(T_{\lambda,\delta})- \lfloor\psi_{a,\lambda,x}(T)\rfloor_{4\pi} \big| \psi_{a,\lambda,x}(T) -  \lfloor\psi_{a,\lambda,x}(T)\rfloor_{4\pi} = \eta \big) \leq xe^{-\frac{\beta}{8}T} +(\hat M +1)\delta.
\end{align*}
Therefore 
\begin{equation}
P\big(\psi_{a,\lambda,x}(T_{\lambda,\delta})-\psi_{a,\lambda,x}(T)>\sqrt{\delta}\big| \psi_{a,\lambda,x}(T) -  \lfloor\psi_{a,\lambda,x}(T)\rfloor_{4\pi} = \eta \big) \leq \frac{ x}{\sqrt{\delta}}e^{-\frac{\beta}{8}T}+\frac{(\hat M +1)}{\sqrt{\delta}}\delta.
\label{eq:tailhead}
\end{equation}
The restriction $T \ge - \frac{8}{\beta} \log \delta$ gives us the desired bound on this region.

For the case $\eta<0$ let $\hat \psi_{a,\lambda,x}\ed \psi_{a,\lambda,-x}$ with initial condition $\hat\psi_{a,\lambda,x}(T)= \eta$ then for $t\ge T$ and $\hat \psi_{a,\lambda,x}$ coupled with $\psi_{a,\lambda,x}$ through the noise term we get
\begin{equation}
\hat \psi_{a,\lambda,x}(t) \le \psi_{a,\lambda,x}(t) - \lfloor \psi_{a,\lambda,x}(T)+2\pi \rfloor.
\end{equation}
The same arguments used for the case $\eta>0$ give us 
\[
P\big(\lfloor \psi_{a,\lambda,x}(T)+2\pi\rfloor_{4\pi}-\hat \psi_{a,\lambda,x}(T_{\lambda,\delta})>\sqrt{\delta} \big) \leq \frac{ x}{\sqrt{\delta}}e^{-\frac{\beta}{8}T}+(\hat M +1)\sqrt{\delta},
\]
which is sufficient to complete the bound for the interval $[T,T_{\lambda,\delta}]$.

We now turn to the stretch $[T_{\lambda,\delta},\infty)$. First observe that at time $T_{\lambda,\delta}$ we have $(\lambda+x) e^{-\frac{\beta}{8}T_{\lambda,\delta}} = \frac{1}{\delta} + \frac{x}{\lambda \delta}$. We return to the original diffusions $\varphi_{a,\lambda+x}$ and $\varphi_{a,\lambda}$. We show that for $\eta < \sqrt{\delta}$ 
\begin{equation}
\label{eq:tailtail}
P\left(\lim_{t\to \infty} \lfloor \varphi_{a,\lambda+x}(t)\rfloor_{4\pi}-\lfloor \varphi_{a,\lambda}(t)  \rfloor_{4\pi} \ne \lfloor\psi_{a,\lambda,x}(T_{\lambda,\delta})\rfloor_{4\pi} \bigg|\ \psi_{a,\lambda,x}(T_{\lambda,\delta}) =\eta \mod 4\pi \right) \ge C(\lambda,\delta),
\end{equation}
where for any $\eps>0$ we may choose $\lambda$ large enough and $\delta$ small enough so that  $C(\lambda,\delta)>1 - \eps$.

To start we note that this problem is equivalent to studying the diffusions restarted at $T_{\lambda,\delta}$, that is $\tilde \varphi_{a,\frac{1}{\delta}+ \frac{x}{\lambda \delta}}$ and $\tilde \varphi_{a,\frac{1}{\delta}}$ satisfying (\ref{phi}) with initial conditions $\tilde \varphi_{a,\frac{1}{\delta}+ \frac{x}{\lambda \delta}}(0) = \varphi_{a,\lambda+x}(T_{\lambda,\delta})\mod 4\pi$ and $\tilde \varphi_{a,\frac{1}{\delta}}(0) = \varphi_{a,\lambda}(T_{\lambda,\delta})\mod 4\pi$. 

Before moving forward we make the following observation: For large enough time $S$ (depending on $\delta$, assuming $x/\lambda \le 1$)
\begin{equation}
P\left( \lfloor \tilde \varphi_{a,\frac{1}{\delta}+ \frac{x}{\lambda \delta}}(S)\rfloor_{4\pi} \ne \lfloor \tilde \varphi_{a,\frac{1}{\delta}+ \frac{x}{\lambda \delta}}(\infty)\rfloor_{4\pi} ,  \lfloor \tilde \varphi_{a,\frac{1}{\delta}}(S)\rfloor_{4\pi} \ne \lfloor \tilde \varphi_{a,\frac{1}{\delta}}(\infty)\rfloor_{4\pi}\right)\ge 1- \eps/2.
\end{equation}
This follows from the fact that $\lfloor \tilde \varphi (t) \rfloor_{4\pi}$ is a monotone increasing function with an almost surely finite limit. Then by using continuous dependence on parameters and initial conditions ($\tilde \varphi_{a,\frac{1}{\delta}+ \frac{x}{\lambda \delta}}(0) -\tilde \varphi_{a,\frac{1}{\delta}}(0) =\eta$ with $|\eta|<\sqrt{\delta}$) we also have that for small enough $\delta$ and large enough $\lambda$
\begin{equation}
P\left( | \tilde \varphi_{a,\frac{1}{\delta}+ \frac{x}{\lambda \delta}}(S)- \tilde \varphi_{a,\frac{1}{\delta}+ \frac{x}{\lambda \delta}}(S)|<2\pi \right) \ge 1-\eps/2.
\end{equation}
The last two bounds, together with the final statement in Theorem \ref{M}, complete the proof of the bound in equation (\ref{eq:tailtail}). The bounds from equations (\ref{eq:tailtail}) and (\ref{eq:tailhead}) give (\ref{eq:conditioned}) which completes the proof of the lemma.

\end{proof}

\begin{proof}[Proof of Lemma \ref{lem:psilimit}]
We need the final statement in Theorem \ref{M}. This gives us that for a fixed $\lambda$ and large enough $T$ we will have that for $t\ge T$, $\varphi_{a,\lambda}(t) \in (M_{a,\beta}(\lambda)4\pi+2\pi, M_{a,\beta}(\lambda)4\pi+4\pi)$. Let $k = M_{a,\beta}(\lambda)$ and $\ell = M_{a,\beta}(\lambda+x)$, then we can choose $T$ large enough so that for $t\ge T$
\begin{align*}
\varphi_{a,\lambda}(t) \in (k4\pi+2\pi, (k+1)4\pi), \quad \text{ and } \quad \varphi_{a,\lambda+x}(t) \in (\ell4\pi+2\pi, (\ell+1)4\pi).
\end{align*}
This gives us that for $t\ge T$
\[
4(\ell-k)\pi + 2\pi > \psi_{a,\lambda,x}(t) > 4(\ell-k)\pi- 2\pi.
\]
Therefore $\lim_{t\to \infty} \lfloor (\psi_{a,\lambda,x}(t)+2\pi)/4\pi\rfloor = M_{a,\beta}(\lambda+x)- M_{a,\beta}(\lambda)$, and it is sufficient to study the $\psi_{a,\lambda,x}$ diffusion.
\end{proof}


\section{The central limit theorem}
\label{bessCLT}

The proof of the central limit theorem may be done in a manner similar to the proof for the $\Sineb$ process which was done by Kritchevski, Valk\'o, and Vir\'ag in \cite{KVV}, but here we will get the result as an easy consequence of Proposition \ref{prop:oscillation}.  

\begin{proof}[Proof of Theorem \ref{thm:bessclt}]

First notice that the process $\hat \varphi_{a,\lambda}(t) = \varphi_{a,\lambda}(t+T)$ with $T = \frac{8}{\beta} \log (\lambda)$ satisfies the same SDE (\ref{phi}) with $\lambda = 1$ with a random initial condition.  Since the equation is $4\pi$ periodic in the $\varphi$ variable and we wish to consider the difference with its limit we may shift the process down so that the initial condition is in the interval $[0,4\pi]$. That is $\hat \varphi_{a,\lambda}(\infty) - \hat \varphi_{a,\lambda}(t)=\tilde \varphi_{a,\lambda}(\infty) - \tilde \varphi_{a,\lambda}(t)$ where $\tilde \varphi_{a,\lambda}(t) = \hat \varphi_{a,\lambda}(t)- \lfloor \hat \varphi_{a,\lambda}(0)\rfloor_{4\pi}$. Here we use $\lfloor \cdot \rfloor_{4\pi}$ to denote rounding down to the next multiple of $4\pi$ as in section 3.  From this we get that 
\[
\frac{\varphi_{a,\lambda}(\infty) - \varphi_{a,\lambda}(T)}{\sqrt{\log \lambda}} \to 0
\]
in distribution and hence in probability.  Because of this it is sufficient to consider the weak limit of 
\[
\frac{\varphi_{a,\lambda}(T) - 8\lambda}{4\pi \sqrt{\log \lambda}} \quad \text{ as } \lambda \to \infty.
\]
Written in its integrated form (and dropping the $a,\lambda$ subscripts), the SDE for $\varphi_{a,\lambda}$ gives us
\begin{align}
\label{bbb}
\varphi(T) - 8\lambda  +8 - 2\pi  & = \frac{\beta}{2} (a+1/2) \int_0^T \sin \big( \frac{\varphi}{2}\big) dt +\int_0^T \frac{\sin (\varphi)}{2} dt + 2 \int_0^T\sin \big(\tfrac{\varphi}{2}\big) dB_t.
\end{align}
We will show that when scaled down by $\sqrt{\log \lambda}$ the first two terms vanish in the limit, then show that the martingale term has the appropriate variance. An application of Proposition \ref{prop:oscillation} gives that the expected value of the first two integrals is finite for all $\lambda$, and so when scaled down by $\sqrt{\log \lambda}$ we get convergence to 0 in probability.

We now turn our attention to the last remaining term in (\ref{bbb}).  We rewrite this as
\[
\frac{1}{\sqrt \lambda} 2\int_0^T \sin \big( \tfrac{\varphi}{2}\big)dB_t = \hat B \left( \frac{1}{\lambda} 4 \int_0^T \sin^2 \big( \tfrac{\varphi}{2}\big)dt\right)= \hat B \left(\frac{16}{\beta} - \frac{2}{\log \lambda}\int_0^T \cos \big( \tfrac{\varphi}{2} \big) dt   \right)
\]
for some standard Brownian motion $\hat B_t$ which depends on $\lambda$. By Proposition \ref{prop:oscillation} this final integral term goes to $0$ in probability. Therefore
\[
P\left(\left|  \hat B \left(\frac{16}{\beta} - \frac{2}{\log \lambda}\int_0^T \cos \big( \tfrac{\varphi}{2} \big) dt   \right) - \hat B \left( \frac{16}{\beta}\right) \right| > \eps\right)
\] may be made arbitrarily small. This is enough to give the desired convergence in distribution, and so completes the proof.

\end{proof}


\section{Large Deviations}

In \cite{DHBV} a large deviation result was proved for the $\Sineb$ process on growing intervals $[0,\lambda]$. The proof of the large deviation for the $\Bess$ process is similar with a few notable differences. The details of the proof will be largely omitted, but we will give an outline of the proof and fill in the details in the steps where the proof differs significantly from the one for $\Sineb$.

\subsubsection*{Outline of the proof of Theorem \ref{thm:bess}}

The proof of the large deviation principle is done by first proving a large deviation principle for the path of the $\varphi_{a,\lambda}$ diffusion.  We use this together with the contraction principle to prove the large deviation result for the end point of the diffusion. The proof of the LDP for the path begins with the observation that $\lfloor \varphi_{a,\lambda}\rfloor_{2\pi}:= 2\pi \lfloor \varphi_{a,\lambda}/2\pi\rfloor$ is a monotone nondecreasing function for $\lambda >0$. Because of this it is enough to understand the time it take $\varphi_{a,\lambda}$ to traverse an interval of the form $[2\pi k, 2\pi (k+1)]$.  Bounds on these travel times will be done using a Girsanov change of measure and for various reasons are easier to compute when the $\frac{\beta}{4}\lambda e^{-\frac{\beta}{4}t}dt$ term is replaced with a constant (or piecewise constant) drift. What follows is a more detailed outline of how to prove Theorem \ref{thm:bess}.

\subsubsection*{\bf Step 1: An LDP for a modified diffusion}

To start we define a diffusion with no time dependence in the drift. Let $\tilde \varphi_{a,\lambda}$ satisfy
\begin{equation}
d \tilde \varphi_{a,\lambda}= \frac{\beta}{2} (a+\tfrac{1}{2}) \sin \big( \frac{\varphi_{a,\lambda}}{2}\big) dt +  \lambda dt + \frac{\sin \varphi_{a,\lambda}}{2} dt + 2 \sin \big(\tfrac{\varphi_{a,\lambda}}{2}\big) dB_t,  \label{eq:tildephi}
\end{equation}
with $\tilde \varphi_{a,\lambda}(0)=2\pi.$ We prove a path large deviation principle for this diffusion on finite time intervals. To prove this we take the following steps:

\begin{enumerate}

\item We can see that when $\tilde \varphi_{a,\lambda}$ is a multiple of $2\pi$ all of the terms vanish except for the $\lambda dt$ term. From this we get that $\lfloor \tilde \varphi_{a,\lambda}\rfloor_{2\pi}$ is nondecreasing when $\lambda >0$.  Define $\tau_k = \inf_{t}\{\tilde \varphi_{a,\lambda}(t)=2\pi (k+1)\}- \inf_{t}\{\tilde \varphi_{a,\lambda}(t)=2\pi k\}$ to be the travel time for the interval $[2\pi k, 2\pi(k+1)]$. Then we will get that for all $k\in \mathbb{N}$, $\tau_{2k}\ed \tau_2$ and $\tau_{2k+1} \ed \tau_3$. We make use of the same change of variables and Girsanov arguments as the $\Sineb$ to get: for $A<1$, then for $\tau= \tau_2, \tau_3$ we have 
\begin{align*}\label{expmbnd}
E e^{ \frac{\lambda^2 A}{8} \tau- \frac{\lambda \tau}{4}(|A|\wedge  \sqrt{|A|})(1+\tfrac{\beta}{2}(a+\tfrac{1}{2}))} \leq  e^{- \lambda \II(A)}.
\end{align*}
Let $t_A=4 K(A)$ and fix  $0<\eps<|t_A-2\pi|$, then
\begin{align*}
P(\lambda \tau \in [t_A- \varepsilon,t_A+\varepsilon] )&\geq C(\varepsilon,\lambda,A) e^{ -\lambda(\II(A)+\tfrac{A t_A}{8}) -  \lambda \frac{|A|\eps}{8} -\lambda |A|(t_A+\varepsilon)(1+\tfrac{\beta}{2}(a+\tfrac{1}{2}))},
\end{align*}
where $\lim\limits_{\lambda\to \infty} C(\eps, \lambda, A)=1$  for fixed $a, \eps$. See Proposition 8 in \cite{DHBV} for the idea of the proof.

\item From the jump bounds we get the following estimates for comparing the diffusion with a linear path: There exist a constant $c$ so that for $\lambda>2$ we have 
\begin{align}
e^{-\lambda^2 t \ldp(q)+\lambda c(t+1)( \ldp(q)+1)}\ge \begin{cases}
  \, P(\lceil \wt \alpha_\lambda(t)\rceil_{2\pi} \ge q t \lambda) &\qquad\textup{if $q>1$,}\\[4pt]
  \, P(\lfloor \wt \alpha_\lambda(t)\rfloor_{2\pi} \le q t \lambda)& \qquad \textup{if $0<q<1$.}
  \end{cases} 
\end{align}
Moreover, there are absolute constants $c_0,c_1$ so that if $qt\lambda, q$ and $\lambda q \log q$ are all bigger than $c_0$ then
\begin{align}\label{eq:uptail}
 P(\lceil \tilde \varphi_{a,\lambda}(t)\rceil_{2\pi} \ge q t \lambda)\le e^{-c_1 \lambda^2 t \,q^2 \log q}.
\end{align}

\item The bounds from the previous step may be used to derive a path deviation result for $\tilde \varphi_{a,\lambda}$.

\begin{theorem}\label{thm:tildepath}
Fix $T>0$ and let $\tilde  \varphi_{a,\lambda(t)}$ be the process defined in (\ref{eq:tildephi}). Then the sequence of rescaled processes $(\tfrac{\tilde  \varphi_{a,\lambda(t)}}{\lambda}, t\in [0,T])$ satisfies a large deviation principle on $C[0,T]$ with the uniform topology with scale $\lambda^2$ and good rate function $\Pathtilde$. The rate function is defined as 
\begin{equation}
\Pathtilde(g)=\int_0^T \ldp\left(g'(t)\right) dt\nonumber
\end{equation}
in the case where $g(0)=0$ and $g$ is absolutely continuous with non-negative derivative $g'$, and  $\Pathtilde(g)=\infty$ in all other cases.
\end{theorem}

\end{enumerate}

\subsubsection*{\bf Step 2: From a path deviation for $\tilde \varphi_{a,\lambda}$ to one for $\varphi_{a,\lambda}$}

 We use the path deviation on $\tilde \varphi_{a,\lambda}$, together with some bounds on the behavior of the diffusion for large $t$ to get the following path diffusion for $\varphi_{a,\lambda}$.

\begin{theorem}
\label{thm:pathbess}
Fix $\beta>0$ and let $\varphi_{a,\lambda}(t)$ be the process defined in (\ref{phi}) with $a>-1$. Then the sequence of rescaled processes $(\tfrac{\varphi_{a,\lambda}(t)}{ \lambda}, t\in [0,\infty))$ satisfies a large deviation principle on $C[0,\infty)$ with scale $\lambda^2$ and good rate function $\Pathbess$. The rate function $\Pathbess$ is defined as 
\begin{equation}
\Pathbess(g) = \int_0^\infty \mathfrak{h}^2(t) \ldp \left( g'(t)/\mathfrak{h}(t)\right)dt, \quad \textup{with} \quad  \mathfrak{h}(t)=\mathfrak{h}_\beta(t)=\beta e^{-\tfrac{\beta}{8}t} \notag
\end{equation}
in the case where $g(0)=0$ and $g$ is absolutely continuous with non-negative derivative $g'$. In all other cases  $\Pathbess(g)$ is defined as $\infty$. 
\end{theorem}

\begin{enumerate}
\item Upper bound

\noindent{\it Approximation 1}: Truncation
	
	Fix $T>0$, the value of which will go to infinity later. Define 
	\[
	\varphi_{a,\lambda}^{(1)} = \varphi_{a,\lambda}(t) \ind(t\le T)+\Big(\varphi_{a,\lambda}(T)+ \lambda(e^{-\frac{\beta}{8}T}- e^{-\frac{\beta}{8}t})\Big)\ind(t>T).
	\]
	Then for $T$ sufficiently large (not depending on $\lambda$), $\limsup_{\lambda\to \infty} \frac{1}{\lambda^2} \log P(\|\varphi_{a,\lambda}^{(1)}- \varphi_{a,\lambda}\| \ge \delta \lambda)  \le c_1 T \delta^2$. The proof of this for the bulk is done in Proposition 13 of \cite{DHBV} with the exception of the very last bound used, which bounds $P(\varphi_{a,\lambda}(\infty)- \varphi_{a,\lambda}(T\lambda) \ge \delta \lambda/2)$. This bound is replaced by Proposition \ref{prop:phitail} which will be given below. 
	
	\bigskip
	
\noindent{\it Approximation 2}: Piecewise constant drift
	
	Define a piecewise constant function that approximates the function $f(t) = \frac{\beta}{8}e^{-\frac{\beta}{8}t}$ by 
	\[
	f_N(t) = f(Ti/N), \quad t\in [Ti/N, T(i+1)/N)
	\]
	Suppose $\varphi_{a,\lambda, N}(t)$ is a diffusion that satisfies the same SDE as $\varphi_{a,\lambda}$ except where the $f(t)dt$ has been replaced by a $f_N(t)dt$ term. Then define
	\[
	\varphi_{a,\lambda}^{(2)}(t) = \varphi_{a,\lambda,N}(t) \ind (t\le T) + (\varphi_{a,\lambda,N}(T)+  \lambda(e^{-\frac{\beta}{8}T}- e^{-\frac{\beta}{8}t})\Big)\ind(t>T).
	\]
	Then 
	\[
	\lim_{\lambda\to \infty} \frac{1}{\lambda^2} \log P(\|\varphi_{a,\lambda}^{(1)}- \varphi_{a,\lambda}^{(2)}\| \ge \delta \lambda)  \le D \left( \frac{\beta T}{8N}\right)^2 T \II \left( \frac{\delta 8N}{\beta T^2}\right).
	\]
	The method of proof is the same as in \cite{DHBV}, but there bound used in line (56) does not hold for the $\tilde \varphi_{a,\lambda}$ diffusion and is replaced by Lemma \ref{lemma:psibnd} given below.

\bigskip	
	
\noindent{\it Approximation 3}: Piecewise constant path 
	
	Let $\pi_{MN}$ be the projection of a path onto a piecewise linear path defined by 
	\[
	(\pi_{MN} g)(Ti/(MN)) = \lfloor g(Ti/(MN))\rfloor_{2\pi}, \qquad \text{ for } t = Ti/(MN)
	\]
	and linear in between these values. Define
	\[
	\varphi_{a,\lambda}^{(3)}(t) =  \pi_{MN}\varphi_{a,\lambda, N}(t)\ind(t\le T)+ (\pi_{MN} \varphi_{a,\lambda,N}(T)+  \lambda(e^{-\frac{\beta}{8}T}- e^{-\frac{\beta}{8}t})\Big)\ind(t>T).
	\]
	This approximation may be treated in the same manor as $\alpha_\lambda^{(3)}$ defined in \cite{DHBV} for both determining the probability of it being close to some particular path, as well as the probability that it is close to $\varphi_{a,\lambda}^{(2)}$.

\item Lower bound

The proof of the lower bound uses similar ideas. We will essentially reuse approximations $\varphi_{a,\lambda}^{(1)}$ and $\varphi_{a,\lambda}^{(2)}$. We show that $\varphi_{a,\lambda}^{(2)}$ stays close to a particular path by making use of the path deviation result on $\tilde \varphi_{a,\lambda}$. The approximation bounds to show that this is sufficient are the same as those in the upper bound. For more on the argument we refer the reader to \cite{DHBV}.

\end{enumerate}

\subsubsection*{\bf Step 3: From a path deviation to the endpoint}

The final step in the proof of Theorem \ref{thm:bess} is to use the contraction principle to go from the path deviation to a large deviation result on the end point. We use the existing analysis in section 7 of \cite{DHBV} and a relationship between the $\Bess$ and $\Sineb$ rate functions to draw our conclusion. Consider the $\Pathbess$ rate function given in Theorem \ref{thm:pathbess}, we apply the change of variables $x = (1-e^{-\frac{8}{\beta}t})$ to get the modified rate function
\[
 \Pbess (\tilde g) = 8 \beta \int_0^1 (1-x)\ldp(\tilde g'(x)/8) dy
\]
where $\tilde g(x) = g(- \tfrac{8}{\beta}\log(1-x))$.  The path large deviation rate function for the $\Sineb$ process \cite{DHBV} is related to this one by
\[
 \Pbess (g) = 32 \Psine (g/8).
\] 
Lastly note that we want to optimize over $g$ with endpoint $4\pi \rho$ where in the $\Sineb$ case we had an endpoint of $2\pi  \rho$.  
This gives us that $\bessldp (\rho) = 32 \sineldp(\rho / 4)$.

\subsubsection*{The details for step 2 part 1 approximations 1 and 2}


The following proposition replaces a tail bound used for the $\alpha_\lambda$ diffusion that does not exist for the $\varphi_{a,\lambda}$ diffusion because of the diffusion dependent drift terms.

\begin{proposition}
\label{prop:phitail}
Let $\varphi_{a,\lambda}$ be the diffusion defined in (\ref{phi}), $T>0$ and $\epsilon>0$ fixed, then 
\begin{equation}
\lim_{\lambda \to \infty} \frac{1}{\lambda^2} \log P(\varphi_{a,\lambda}(\infty) - \varphi_{a,\lambda}(T \lambda) \geq \epsilon \lambda) \leq -c_\beta T \epsilon.
\end{equation}
where $c_\beta$ is some explicitly computable constant depending only on $\beta$.
\end{proposition}

\begin{proof}
Recall that for a fixed $T>0$ the diffusion $\varphi_{a,\lambda}(t+T)$ satisfies the same stochastic differential equation as $\varphi_{a,\tilde \lambda}(t)$ with $\tilde \lambda = \lambda e^{-\frac{\beta}{8}T}$ and initial condition $\varphi_{a,\tilde \lambda}(0)= \varphi_{a,\lambda}(T)$. In particular to study $\varphi_{a,\lambda}(\infty)- \varphi_{a,\lambda}(\lambda T) $ it is sufficient to study the $\varphi_{a,\hat \lambda}$ diffusion with $\hat \lambda = \lambda e^{-\frac{\beta}{8}\lambda T}$ and random initial condition in $[0,4\pi]$. 

We will start with the case $a\ge -1/2$. Recall the diffusion $X_{1}$ from section 2 defined in (\ref{eq:x1}). We are working with the singular value process, therefore we work with a diffusion where $\sqrt \lambda$ has been replaced with $\lambda$. In particular we study $X_\lambda$ which satisfies
\[
dX_\lambda(t) = \left( \frac{\beta}{4}(-a - \frac{1}{2})+ \frac{\beta}{2} \lambda e^{-\frac{\beta}{8}t} \cosh X_{\lambda}(t)\right)dt - dB_t, \qquad X_\lambda(0)=-\infty.
\] 
Notice that in order to have $\varphi_{a,\lambda}(\infty)  \ge k 4\pi$ we must have that it crosses an interval of the type $ [\ell 4\pi-2\pi , \ell 4\pi]$ at least $k$ times. On this interval we can do a change of variables so this is equivalent to $X_{1, \lambda_{\ell}}$ exploding for $\ell \le k$ where $\lambda_\ell$ is $\lambda e^{-\frac{\beta}{4}t_\ell}$ where $t_\ell$ is the hitting time of $\ell 4\pi-2\pi$. Notice that as $\lambda$ decreases the likelihood of explosion decreases, so we may use the following bound
\begin{align}
P(\varphi_{a,\lambda}(\infty)- \varphi_{a,\lambda}(\lambda T) \ge \lambda \eps) & \le P( \varphi_{a,\hat \lambda}(\infty) \ge \lambda\eps -4\pi)\notag \\
& \le \Big[ P\big( X_{\hat \lambda} \text{ does not explode}\big)\Big]^{\lfloor \frac{\lambda \eps}{4\pi}\rfloor-1}\label{eq:bnd3}.
\end{align} 
In order to study the probability that $X_{\hat \lambda}$ explodes we define $Z = X_{\hat \lambda}+ B_t$, then $Z$ explodes at the same time as $X_{\hat \lambda}$ and satisfies the differential equation
\[
Z'(t) = - \frac{\beta}{4}(a+\frac{1}{2}) + \frac{\beta}{2} \hat \lambda e^{-\frac{\beta}{8}t} \cosh (Z(t)-B_t), \qquad Z(0)=-\infty.
\]
In its integrated form we get 
\[
Z(t) = - \frac{\beta}{4}(a+\frac{1}{2})t + \frac{\beta}{2} \hat \lambda \int_0^t e^{-\frac{\beta}{8}s}\cosh(Z(s)-B_s)ds
\]
We now observe that $\cosh(Z-B) \le 2\cosh Z \cosh B$, therefore since the remaining drift term is less than or equal to 0 for $a\ge -1/2$ and we get that 
\begin{align*}
P( Z(t) \text{ crosses} [-M, M] ) &\le P\left( \frac{\beta}{2} \hat \lambda \int_0^\infty e^{-\frac{\beta}{8}s}\cosh Z(s) \cosh B_s  \ind_{\{Z(s) \ge -M\}} ds \ge 2M \right) \\
& \le P\left( \frac{\beta}{4} \hat \lambda \cosh M \int_0^\infty e^{-\frac{\beta}{8}s+B_s} +e^{-\frac{\beta}{8}s-B_s}   ds \ge 2M\right)\\
& \le 2 P\left( \frac{\beta}{4} \hat \lambda \cosh M \int_0^\infty e^{-\frac{\beta}{8}s+B_s} ds \ge M\right),
\end{align*}
where the final line comes because $B_s$ and $-B_s$ have the same distribution. We now break this into two pieces first we consider the integral from $0$ to $\lambda T$, and then we consider the integral from $\lambda T$ to $\infty$. For $\lambda$ large enough so that $\frac{2 M}{\beta \hat \lambda \cosh M} \ge 1$ we get 
\begin{align}
P\Bigg( \frac{\beta}{4} \hat \lambda \cosh M \int_0^{\lambda T} &e^{-\frac{\beta}{8}s+B_s} ds \ge \frac{M}{2}\Bigg)  \le P\left( \sup_{t\in [0, \lambda T]} e^{B_t} \ge \frac{2 M}{\beta \hat \lambda \cosh M}\right)\notag \\
& \le \frac{2}{\sqrt{2\pi}}\exp\left(- \frac{\beta}{16} \lambda T + \frac{\beta}{8}\log \frac{4 M}{\lambda \log M}+ \frac{1}{2\lambda T} \log^2 \frac{4 M}{\lambda \log M}\right). \label{eq:bnd1}
\end{align}
Now for the other integral we look at the probability that $B_t$ stays below the line $\frac{\beta}{16} t$ for $t \ge \lambda T$. If this happens then for large enough $\lambda$ we get
\[
\int_{\lambda T}^\infty e^{-\frac{\beta}{8}s+ B_s} ds \le \frac{16}{\beta} e^{- \frac{\beta}{16}\lambda T} \le \frac{2 M}{\beta \hat \lambda \cosh M}.
\]
Using this we get the following bounds on our second integral term:
\begin{align}
P\bigg( \int_{t\lambda}^\infty e^{-\frac{\beta}{8}s+B_s} & ds \ge \frac{2 M}{\beta \hat \lambda \cosh M}\bigg)    \le  P\bigg(\exists\  t\in [\lambda T, \infty) \text{ such that } B_t \ge \tfrac{\beta}{16}t \bigg)\notag \\
& \le P\big(\exists\  t\in [\lambda T, \infty) \text{ such that } B_t \ge \tfrac{\beta}{16}t, B_{\lambda T}< \tfrac{\beta}{32}\lambda T\big)+P(B_{\lambda T}\ge \tfrac{\beta}{32}\lambda T)\notag \\
& \le P\big( \exists \ t\in [0,\infty) \text{ such that } B_t \le \tfrac{\beta}{16}t+ \tfrac{\beta}{32} \lambda T) + \frac{1}{\sqrt{2\pi}}\exp\Big(-\tfrac{1}{2}(\tfrac{\beta}{32})^2\lambda T\Big)\notag \\
& \le \exp \big(-( \tfrac{\beta}{16})^2 \lambda T\big)+ \frac{1}{\sqrt{2\pi}}\exp\big(-\tfrac{1}{2}(\tfrac{\beta}{32})^2\lambda T\big). \label{eq:bnd2}
\end{align}
Putting the bounds from (\ref{eq:bnd1}) and (\ref{eq:bnd2}) together into the bound in line (\ref{eq:bnd3}) completes the proof of the proposition for $a\ge -1/2$. For $a<-1/2$ we may do a similar analysis using the $X_2$ diffusion in (\ref{eq:x2}).
\end{proof}



We now focus on the second approximation $\varphi^{(2)}_{a,\lambda}(t)$. Observe that the diffusion $\psi_{a,\lambda,N}(t)= \varphi_{a,\lambda,N}(t)-\varphi_{a,\lambda}(t)$ that appears will be stochastically dominated by an diffusion that satisfies the SDE
\begin{align}
d \tilde \psi_{\lambda,\tilde \lambda} &= \frac{\beta}{2}(a+1/2) \im \left[e^{i \frac{ \varphi_{a,\lambda}}{2}}\left(e^{-i \frac{\tilde \psi_{\lambda,\tilde \lambda}}{2}}-1\right)\right]dt + \frac{1}{2}\im \left[e^{i \varphi_{a,\lambda}}\left(e^{-i \tilde \psi_{\lambda,\tilde \lambda}}-1\right)\right]dt\notag \\
&\hspace{1.5cm}+ \tilde \lambda dt  +\im \left[e^{i \frac{ \varphi_{a,\lambda}}{2}}\left(e^{-i \frac{\tilde \psi_{\lambda,\tilde \lambda}}{2}}-1\right)\right]dB_t, \label{eq:tildepsi}
\end{align}
with initial condition $\tilde \psi(0)=0$ for $\tilde \lambda= \lambda \frac{\beta T}{8N}$. To prove that $\tilde \psi_{\lambda, \tilde \lambda}$ has the desired behavior we need to revert to working with jump times. In particular we can show:

\begin{proposition}
\label{prop:taubnd2}
Let $\sigma_k =\inf_t\{ \tilde \psi_{\tilde \lambda,\lambda}(t)=4\pi (k+1)\}- \inf_t\{\tilde \psi_{\tilde \lambda,\lambda}(t)=4\pi k\} $  and fix $A<0, T>0$ and $0< \delta < 1/4$, then there exist constants $C_1$ and $C_2$ depending only on $T, \delta$ and $A$ such that for $\tilde \lambda \ge 1$
\begin{align}\label{expmbnd}
E e^{ 8(1+2\delta)^2\tilde \lambda^2 A \sigma_k- 2\tilde \lambda (1+2\delta) \sigma_k(|A|\wedge  \sqrt{|A|})(2+8\delta)} \leq  e^{-8\tilde \lambda (1+2\delta) \II(A)} + C_1e^{-C_2\lambda^2 (T+1/T)}.
\end{align}
\end{proposition}

Following the proof of Lemma 10 in \cite{DHBV} this can be used to show the following bound. 

\begin{lemma}
\label{lemma:psibnd}
Let $\tilde \psi_{\tilde \lambda, \lambda}$ be defined as in (\ref{eq:tildepsi}) with $\tilde \lambda = C\lambda$ for some constant $C$, then for $q>C$ there exists a constant $D$ so that for $\tilde \lambda >2$ we have 
\[
 \lim_{\lambda\to \infty} \frac{1}{\lambda^2} \log P\Big( \lfloor\tilde \psi_{\tilde \lambda, \lambda}(t) \rfloor_{4\pi} \ge q t \lambda\Big) \le 
 -D C^2  t \II(q/C).
\]
\end{lemma}

These two propositions are the only major changes needed to adapts the proof of the large deviation result for $\Sineb$ given in \cite{DHBV} to the $\Bess$ process.


\begin{proof}[Sketch of proof of Proposition \ref{prop:taubnd2}]

The proof of the jump time bound for $\sigma_k$ is done through a coupling argument. Let $T_k = \sigma_1+\cdots + \sigma_{k-1}$, we make the change of variables $\tilde \psi_{\tilde \lambda, \lambda} (T_k+t) = 8 \arctan (e^{Y(t)})$ on the interval $[4\pi k, 4\pi(k+1))$. We get that (dropping the subscripts) $Y$ satisfies the stochastic differential equation
\begin{align*}
dY& = \frac{\tilde \lambda}{4}\cosh Y dt - \frac{1}{8} \tanh Y dt- \frac{\beta}{4}(a+\tfrac{1}{2})\Big[ \sin \Big( \frac{\varphi}{2}\Big)\sech Y- \cos \Big( \frac{\varphi}{2}\Big) \tanh Y\Big]dt\\
& \hspace{1.25cm} - \frac{1}{8}\Big[2\cos \varphi \sech^2Y \tanh Y+ 2\sin \varphi \sech Y \tanh^2 Y- \cos \varphi \tanh Y\Big]dt\\
& \hspace{2.25cm}   + \frac{1}{2}\Big[ \sin \Big( \frac{\varphi}{2}\Big)\sech Y- \cos \Big( \frac{\varphi}{2}\Big) \tanh Y\Big]dB_t.
\end{align*}
with initial condition $Y(0)= -\infty$ and $Y$ explodes at the hitting time $\sigma_k$. These oscillatory integrals are not the same as the ones that appear in Proposition \ref{prop:oscillation}, but they are amenable to the same type of analysis. 
We can check that the quadratic variation of this process is 
\[
d[Y]_t = \frac{1}{8}dt+\frac{1}{8} \Big[ 2\cos \varphi \tanh Y + 2 \sin \varphi \sech Y \tanh Y- \cos \varphi \Big]dt.
\]
Let $ \os_{Y,t}$ denote the finite variation terms involving $\sin(c\varphi)$ or $\cos(c\varphi)$. We can check that $P(|\int_0^T  \os_{Y,t}dt|\ge \delta)\le \exp\left[ - C \delta^2 \lambda^2 /T\right]$ for some constant $C$. For convenience we write 
\[
\frac{1}{2}\Big[ \sin \Big( \frac{\varphi}{2}\Big)\sech Y- \cos \Big( \frac{\varphi}{2}\Big) \tanh Y\Big]dB_t = \frac{1}{2} \Big[ \frac{1}{\sqrt 2} + g(Y, \varphi)\Big] dB_t.
\]
Let $W_t =\int_0^t g(Y,\varphi)dB_s$ and consider the diffusion $\hat Y$ given by $\hat Y_t = Y_t - W_t$. Then 
\begin{align*}
d\hat Y &= \frac{\tilde \lambda}{4} \left( \cosh \hat Y \cosh W_t - \sinh \hat Y \sinh W_t \right) dt+ \frac{1}{2\sqrt 2} dB_t - \frac{1}{8} \tanh Y dt+ \os_{Y,t}dt
\end{align*}
with $\hat Y(0)=-\infty$.
Notice that $W_t$ is finite almost surely, and so $\hat Y$ explodes at the same time as $Y$. This explosion time is $\sigma_k$. Therefore proving bounds on the explosion time of $\hat Y$ is enough.

We can choose $\delta$ small enough, so that we get 
$Z_t^- \le \hat Y_t \le Z_t^+
$
where 
\begin{align}
\label{eq:zpm}
dZ^\pm = \frac{\tilde \lambda}{4}(1\pm 2\delta)\cosh Z^\pm dt \pm (\tfrac{1}{8}+\delta) dt + \frac{1}{2\sqrt 2}dB_t\qquad Z^\pm(0)=-\infty.
\end{align}
The explosion time of $\hat Y_t$ will be bounded between the explosion times of $Z^{\pm}$ (on the set where the oscillatory integrals are small). Now the $Z^{\pm}$ diffusions may be treated using the same methods as Proposition 8 in \cite{DHBV}.

\end{proof}


{\small

\bibliographystyle{abbrv}
\bibliography{rmt}

}

\end{document}